\documentclass[12pt]{amsart}
\usepackage{amssymb,latexsym,amsmath,amsthm,amscd}
\usepackage{setspace}
\usepackage{verbatim}
\usepackage[active]{srcltx}
\usepackage[all]{xy} \xyoption{arc}
\usepackage[left=3cm,top=2cm,right=3cm,bottom = 2cm]{geometry}
\usepackage{graphicx}
\usepackage[usenames,dvipsnames]{color}
\usepackage[linkcolor=darkgreen]{hyperref}

\theoremstyle{plain}
\newtheorem{thm}{Theorem}[section]

\newtheorem{conj}[thm]{Conjecture}

\theoremstyle{definition}
\newtheorem{dfn}[thm]{Definition}
\newtheorem{ex}[thm]{Example}

\theoremstyle{remark}
\newtheorem{rmk}[thm]{Remark}

\makeatletter
\renewcommand*{\eqref}[1]{%
  \hyperref[{#1}]{\textup{\tagform@{\ref*{#1}}}}%
}
\makeatother

\newcommand{\cA}{\mathcal{A}}

\newcommand{\cL}{\mathcal{L}}

\newcommand{\cO}{\mathcal{O}}

\newcommand{\cV}{\mathcal{V}}

\newcommand{\veps}{\varepsilon}

\DeclareMathOperator{\uhp}{\mathcal{H}}

\DeclareMathOperator{\rank}{rank}

\DeclareMathOperator{\Endo}{End}

\DeclareMathOperator{\im}{im}

\DeclareMathOperator{\Pic}{Pic}

\DeclareMathOperator{\GL}{GL}

\DeclareMathOperator{\SL}{SL}
\DeclareMathOperator{\PSL}{PSL}

\newcommand*{\df}{\mathrel{\vcenter{\baselineskip0.5ex \lineskiplimit0pt
                     \hbox{\scriptsize.}\hbox{\scriptsize.}}} =}

\providecommand{\abs}[1]{\left\lvert#1\right\rvert}

\providecommand{\twomat}[4]{\left(\begin{matrix}#1&#2\\#3&#4\end{matrix}\right)}
\providecommand{\stwomat}[4]{\left(\begin{smallmatrix}#1&#2\\#3&#4\end{smallmatrix}\right)}
\providecommand{\twovec}[2]{\left(\begin{matrix}#1\\#2\end{matrix}\right)}

\newcommand{\CC}{\mathbf{C}}

\newcommand{\ZZ}{\mathbf{Z}}
\newcommand{\PP}{\mathbf{P}}
\newcommand{\RR}{\mathbf{R}}

\newcommand{\cVbar}{\overline{\cV}}

\newcommand{\Ebar}{\overline{E}}

\newcommand{\floor}[1]{\left \lfloor #1 \right\rfloor}

\DeclareMathOperator{\Herm}{Herm}
\DeclareMathOperator{\bpartial}{{\bar\partial}}
\DeclareMathOperator{\btau}{{\bar\tau}}
\DeclareMathOperator{\Higgs}{Higgs}

\DeclareMathOperator{\Gr}{Gr}

\DeclareMathOperator{\res}{res}

\begin{document}
\title[Nonabelian Hodge theory and vector valued modular forms]{Nonabelian Hodge theory and vector valued modular forms}
\author{Cameron Franc and Steven Rayan}
\email{franc@math.usask.ca; rayan@math.usask.ca}
\address{Department of Mathematics \& Statistics, University of Saskatchewan, McLean Hall, 106 Wiggins Road, Saskatoon, SK, CANADA S7N 5E6}
\date{\today}

\begin{abstract}
We examine the relationship between nonabelian Hodge theory for Riemann surfaces and the theory of vector valued modular forms. In particular, we explain how one might use this relationship to prove a conjectural three-term inequality on the weights of free bases of vector valued modular forms associated to complex, finite dimensional, irreducible representations of the modular group. This conjecture is known for irreducible unitary representations and for all irreducible representations of dimension at most $12$. We prove new instances of the three-term inequality for certain nonunitary representations, corresponding to a class of maximally-decomposed variations of Hodge structure, by considering the same inequality with respect to a new type of modular form, called a ``Higgs form'', that arises naturally on the Dolbeault side of nonabelian Hodge theory. The paper concludes with a discussion of a strategy for reducing the general case of nilpotent Higgs bundles to the case under consideration in our main theorem.
\end{abstract}

\maketitle
\tableofcontents

\section{Introduction}
Suppose that $Y$ is a hyperbolic orbifold curve uniformized by the complex upper half plane $\uhp$, so that $Y$ can be identified with a quotient $Y = \Gamma \backslash \uhp$ of $\uhp$ under the action of a discrete group of linear fractional transformations $\Gamma$. Assume that $\Gamma$ is a Fuchsian group of the first kind, so that $Y$ has finite volume and a finite number of cusps. Vector bundles $\cV$ on $Y$ can be pulled back to $\uhp$ and trivialized there --- this trivialization can even be achieved holomorphically for holomorphic bundles, by the Stein property of $\uhp$. Thus, by uniformization, all vector bundles on $Y$ can be described as trivial bundles $\uhp \times \CC^n$ endowed with an action of the fundamental group $\Gamma$:
\[
  \gamma (\tau,v) = (\gamma \tau, c(\gamma,\tau)v).
\]
The function $c \colon \Gamma \times \uhp \to \GL_n(\CC)$ appearing above satisfies the cocycle identity $c(\alpha\beta,\tau) = c(\alpha,\beta\tau)c(\beta,\tau)$. The bundle-sheaf dichotomy says that giving the data of the bundle $\cV$ is equivalent to giving the data of its sheaf of local sections, which are vector valued functions on open subsets of $\uhp$ satisfying the modular transformation law
\[
  F(\gamma\tau) = c(\gamma,\tau)F(\tau).
\]
Thus, in this sense, the study of vector bundles on $Y$ is equivalent to the study of generalized automorphic functions for the fundamental group $\Gamma$.

In the last several decades geometers have been exploring the geometry of Higgs bundles and their relationship with Hodge theory, beginning with \cite{Hitchin,Simpson1,Simpson2}.  Higgs bundles originated in mathematical investigations into gauge theory, specifically the self-dual Yang-Mills equations on a Riemann surface \cite{Hitchin}, now known as the \emph{Hitchin equations}.  The rich geometry inherited by moduli spaces of Higgs bundles has been exploited in a variety of fields: surface group theory (see \cite{Labourie} for a survey); integrable systems (beginning with \cite{HitchinDuke}); mirror symmetry (\cite{HauselThaddeus,KapustinWitten,DonagiPantev2} for instance); and topological recursion and free probability (\cite{DumitrescuMulase}), to name a few areas.  Most spectacularly, Higgs bundles lie at the heart of Ng\^{o}'s proof of the Fundamental Lemma \cite{Ngo} from the Langlands Program, which is perhaps the first glimpse into a truly deep relationship between Higgs bundles and number theory.

The goal of the present paper is to discuss applications of ideas of nonabelian Hodge theory to a question concerning the weights of basis elements of spaces of (weakly holomorphic) vector valued modular forms for the modular group $\Gamma(1) \df \PSL_2(\ZZ)$.  To be more precise, let $Y = \Gamma(1) \backslash \uhp$ and let $X$ stand for its compactification by cusps. Then $X$ can be identified with the weighted projective line $\PP(2,3)$, and the splitting principle holds for holomorphic vector bundles on $X$. If $\rho \colon \Gamma(1) \to \GL_n(\CC)$ denotes a representation of the modular group, then holomorphic modular functions for $\rho$ are global sections of a logarithmic extension $\cV(\rho)$ to the cusp of the natural flat connection on $Y$ associated to $\rho$. If $\cV_{2l}(\rho) \df \cV(\rho) \otimes \cO(l)$, then global sections of $\cV_{2l}(\rho)$ are modular forms of weight $k = 2l$. Let $M_{k}(\rho)$ denote the space of global sections of $\cV_{k}(\rho)$, and define the space of \emph{holomorphic vector valued modular forms for $\rho$} to be
\[
  M(\rho) \df \bigoplus_{k \in 2\ZZ} M_{k}(\rho).
\]
The splitting principle for holomorphic bundles on $X$ implies that $M(\rho)$ is a free module over the ring of classical scalar valued modular forms of level one, which is a polynomial ring in two generators $E_4$ and $E_6$ in degrees $2$ and $3$, respectively (corresponding to Eisenstein series of weights $4$ and $6$) isomorphic with the projective coordinate ring of $\PP(2,3)$. Moreover, if $\cV(\rho)$ decomposes as
\[
  \cV(\rho) \cong \bigoplus_{j=1}^n\cO(r_j),
\]
then $M(\rho)$ has a free-basis of modular forms of weights $k_j = -2r_j$.

Computing the roots $r_j$ above, or equivalently, the weights $k_j$, is a difficult problem in general. When $\rho$ is unitarizable, one can use positivity and Serre duality to give a formula for the roots, but the nonunitarizable case is much more difficult. Note that, if one works with $\SL_2(\ZZ)$ instead of $\PSL_2(\ZZ)$, one encounters difficulties even for representations of finite image: in that case, forms of weight one are mapped to themselves under Serre duality, and the geometric argument collapses. Due to this, even to this day the dimensions of spaces of classical holomorphic congruence modular forms of weight one are poorly understood. See \cite{Duke} and \cite{MichelVenkatesh} for some asymptotic results.

In \cite{FrancMason} it was observed that for irreducible $\rho$, the weights of a free basis for $M(\rho)$ tend to satisfy a three-term inequality. To state this precisely, let $m_k$ denote the multiplicity of forms of (even) weight $k$ among a free basis for $M(\rho)$ over the ring of scalar valued modular forms. Equivalently, it is the multiplicity of $\cO(-\frac 12k)$ appearing in the decomposition of $\cV(\rho)$.
\begin{conj}
\label{3termconj}
  For irreducible $\rho$, the three-term inequality
  \[
  m_k \leq m_{k+2}+m_{k-2}
\]
is satisfied for all $k \in \ZZ$.
\end{conj}
Conjecture \ref{3termconj} was proved in \cite{FrancMason} for irreducible unitary representations of $\SL_2(\ZZ)$  (not just $\PSL_2(\ZZ)$) using positivity, and it was further established for \emph{all} irreducible representations satisfying $\dim \rho \leq 12$ via computation. While Conjecture \ref{3termconj} sheds no new light on the question of modular forms of weight one (in that case, for unitary representations of $\SL_2(\ZZ)$, Conjecture \ref{3termconj} says simply that $\dim M_1(\rho) \leq \dim M_3(\rho)$, since $\dim M_{-1}(\rho) = 0$ for unitary $\rho$), it is a surprising restriction on the sorts of bases that can occur for spaces of modular forms.

It has been known since the work of Narasimhan-Sheshadri \cite{NarasimhanSeshadri} in the case of compact curves, and Mehta-Seshadri \cite{MehtaSeshadri} for punctured curves, that irreducible unitary representations of a Fuchsian group $\Gamma$ correspond to holomorphic bundles on the compactified quotient $X$ with parabolic structures at cusps. A few decades later, Hitchin \cite{Hitchin}, Donaldson \cite{Donaldson} and Simpson \cite{Simpson1,Simpson2} discovered that if one wishes to consider all irreducible representations --- not just the unitary ones --- then the right category to look at is the category of \emph{filtered} or \emph{parabolic Higgs bundles} on $X$. Such an object consists of a filtered, or parabolic, bundle $E$ on $X$ paired with an $\cO_X$-linear endomorphism
\[
  \theta \colon E \to E\otimes\Omega^1_X(S)
\]
where $S = X \setminus Y$ denotes the (possibly empty) set of cusps, and $\theta$ respects the filtered structure on $E$. For higher dimensional spaces there is an additional flatness condition, but for curves this condition is automatic.

The goal of this paper is to explain, in automorphic terms, the constructions of nonabelian Hodge theory that allow one to pass between representations and filtered Higgs bundles, and to apply them to the study of Conjecture \ref{3termconj}. In Section 2 we describe the various filtered objects that arise in the statement of nonabelian Hodge theory, keeping in mind the automorphic description of these objects. Section 3 recalls some facts on vector valued modular forms and gives a precise discussion of Conjecture \ref{3termconj}. Section 4 next describes the nonabelian Hodge correspondence for noncompact curves, following \cite{Simpson2} \footnote{Although \cite{Simpson2} only treats the case of curves, and not orbifolds, Simpson's arguments all take place on the universal cover, and so they generalize readily to the case of orbifolds required in this paper.}, and we introduce the natural space of \emph{Higgs modular forms} $\Higgs(\rho)$ that is the output of nonabelian Hodge theory. When $\Gamma = \PSL_2(\ZZ)$ we use the splitting principle to show that $\Higgs(\rho)$ is a free module of rank $\dim \rho$ over the classical ring of scalar valued modular forms for $\Gamma$, analogously to the module $M(\rho)$ discussed above. In the final Section 5 we prove Conjecture \ref{3termconj} for stable filtered Higgs bundles corresponding to certain variations of Hodge structure:
\begin{thm}
  \label{t:main}
Let $(E,\theta)$ be a stable filtered Higgs bundle corresponding to a variation of Hodge structure, such that $\theta$ is nilpotent and decomposes $E$ into a sequence of line bundles $L_i$ for which $\theta$ shifts the sequence in the following way: $$\theta:L_i\to L_{i-1}\otimes\Omega^1_X(S).$$ Then, Conjecture \ref{3termconj} holds for $E$.
\end{thm}

Note that the representations that correspond to $(E,\theta)$ as in Theorem \ref{t:main} are not unitary, and so Theorem \ref{t:main} represents a new class of cases for which the three-term inequality is known. The rest of Section 5 discusses ideas for using the geometry of the moduli space of Higgs bundles to extend this argument to the general case.

In the final section of the paper, we speculate on how Morse-theoretic techniques associated with the topology of the Higgs bundle moduli space, namely a natural $\mathbf C^*$-action and its flows, can be used to extend Theorem \ref{t:main} to nilpotent filtered Higgs bundles that do not arise from complex variations of Hodge structure.  This program rests on the observation that certain directions of the downward Morse-Bott flow are reorganizations of the line bundles $L_i$ in the statement of the theorem.

This note reports on a talk given at a conference in honour of Geoff Mason that took place at Sacramento State University in 2018, and we thank the organizers for the opportunity to speak on this work. This is fitting, given that Knopp and Mason initiated the general study of vector valued modular forms for any representation of $\SL_2(\ZZ)$ in the sequence of papers \cite{KnoppMason1,KnoppMason2,KnoppMason3,KnoppMason4,KnoppMason5}. The original approach of Knopp-Mason was rooted in classical techniques of analytic number theory and motivated by questions arising in the theory of vertex operator algebras.  This article aims to show that new approaches to these problems can arise by appealing to the correspondence between vector valued modular forms and Higgs modular forms.  It is our general hope that this article will bring together geometers, both differential and algebraic, and automorphic representation theorists in the pursuit of new techniques and tools for solving open problems coming from either community.

\emph{Acknowledgements}: Both authors are grateful for the support provided by NSERC Discovery Grants during the preparation of this manuscript.

\section{Filtered objects}
Let $Y$ be an orbifold curve uniformized by the complex upper half plane $\uhp$ via a map $j \colon \uhp \to Y$. Assume that the fundamental group $\Gamma\subseteq \PSL_2(\RR)$ of this cover is a Fuchsian group of the first kind, so that the compactification $X=Y\cup S$ is obtained by adding a finite, possibly empty, set of cusps $S$ to $Y$.

\subsection{Filtered representations} For each cusp $s \in S$, we let $\Gamma_s\subseteq \Gamma$ denote the stabilizer subgroup of $s$.

\begin{dfn}
  A \emph{filtered representation} of $\Gamma$ consists of a finite dimensional complex representation $\rho \colon \Gamma \to \GL(V)$ and filtrations $F_{\beta,s}V$ for each $s \in S$ indexed by real numbers $\beta$ such that the following axioms are satisfied:
  \begin{enumerate}
  \item (decreasing) $F_{\alpha,s}V \subseteq F_{\beta,s}V$ for all $\alpha \geq \beta$;
  \item (exhaustive) $\cup_{\beta} F_{\beta,s}V = V$ and $\cap_{\beta} F_{\beta,s}V = 0$ for all $s \in S$;
  \item (left continuous) for all $\beta \in \RR$, there exists $\veps > 0$ such that $F_{\beta-\veps,s}V = F_{\beta,s}V$;
  \item (monodromy invariance) $\Gamma_s \cdot F_{\beta,s}V \subseteq F_{\beta,s}V$ for all $\beta \in \RR$.
  \end{enumerate}
\end{dfn}

\begin{rmk}
In the literature on moduli of parabolic bundles, both decreasing and increasing filtrations appear.  This is a matter of convention and we have opted for increasing filtrations in order to match Simpson's choice in \cite{Simpson1}.
\end{rmk}

Note that since $V$ is finite dimensional, condition (2) implies that there exists $A\leq B$ such that $F_{A,s}V = V$ and $F_{B,s}V = 0$. If $(\rho,F)$ is a filtered representation of $\Gamma$ on a vector space $V$, then for each $s \in S$ define
\begin{align*}
  \Gr_{\beta,s}(\rho,F) &\df \frac{F_{\beta,s}V}{\bigcup_{\alpha >\beta} F_{\alpha,s}V},\\
  \Gr_s (\rho,F) &\df \bigoplus_{\beta \in \RR} \Gr_{\beta,s}(\rho,F).
\end{align*}
We call a real number $\beta$ a \emph{jump} of $(\rho,F)$ provided that $\Gr_{\beta,s}(\rho,F) \neq 0$. A jump is characterized by the property that $F_{\beta,s}V \neq F_{\beta+\veps,s}V$ for all $\veps  > 0$.

\begin{dfn}
  The \emph{degree} of a filtered representation $(\rho,F)$ is defined as
  \[
  \deg(\rho,F) \df \sum_{s \in S}\sum_{\beta \in \RR} \beta\dim_\CC(\Gr_{\beta,s}(\rho,F)).
\]
The \emph{slope} of a filtered representation $(\rho,F)$ is the quantity $\mu(\rho,F) = \deg(\rho,F)/\dim_{\CC} V$.
\end{dfn}

\begin{dfn}
  A filtered representation $(\rho,F)$ is said to be \emph{semistable} provided that for every subobject $(\rho',F') \subseteq (\rho,F)$, the inequality $\mu(\rho',F') \leq \mu(\rho,F)$ holds. If the inequality holds strictly whenever $(\rho',F')$ is nontrivial, then $(\rho,F)$ is said to be \emph{stable}.
\end{dfn}

\begin{ex}
  Let $(\rho,F)$ be a filtered representation with $\dim_\CC \rho = 1$. For each $s \in S$, there is one jump $\beta \in \RR$ and $\dim \Gr_{\beta,s}(\rho,F) = 1$. Hence $\deg(\rho,F) = \mu(\rho,F)$ is the sum of the jumps. In this case $(\rho,F)$ is stable since there are no nontrivial sub-objects to test against the stability condition.
\end{ex}

\begin{dfn}
If $s\in S$ is a cusp, then the stabilizer $\Gamma_s$ has two generators. Let $\gamma_s \in \Gamma_s$ denote the unique generator that is conjugate to a matrix $\stwomat 1w01$ where $w$ is a positive real number. The \emph{residue} $\res_s(\rho,F)$ of $(\rho,F)$ at $s$ is the endomorphism $\rho(\gamma_s)$.
\end{dfn}

\subsection{Filtered regular connections}
Closely related to the notion of a filtered representation is that of a filtered connection or filtered flat bundle. To make precise this notion, let $(\rho,F)$ denote a filtered representation and define an action of $\Gamma$ on the trivial bundle $\uhp \times V$ by setting $\gamma(\tau,v) = (\gamma \tau, \rho(\gamma)v)$. This descends to a vector bundle on $Y$ that we denote $\cV(\rho)$. Sections of $\cV$ over $U \subseteq Y$ are identified with functions $f \colon j^{-1}(U) \to \CC^d$ satisfying $f(\gamma\tau) = \rho(\gamma)f(\tau)$ for all $\gamma \in \Gamma$ and $\tau \in j^{-1}(U)$. In particular, global holomorphic sections are holomorphic modular forms of weight zero for $\rho$ that have no conditions imposed at the cusp.

Weakly holomorphic modular forms associated to $\Gamma$ are global sections of the meromorphic extension of $\cV(\rho)$ to the cusp. Since the dagger symbol commonly denotes spaces of weakly holomorphic modular forms in the literature, we will write $\cV_\dagger(\rho)$ for the meromorphic extension. It can be described concretely as follows: each stabilizer $\Gamma_s$ of a cusp $s$ has two possible generators to choose from. Let $\gamma_s$ denote the unique generator such that$$\displaystyle g^{-1}\gamma_sg = \left(\begin{array}{cc}1 & w\\ 0 & 1\end{array}\right)$$for some $g \in \PSL_2(\RR)$, where $w$ is a \emph{positive} real number. Define a new function $f_g(\tau) = f(g\tau)$ and observe that $f_g(\tau+w) = \rho(\gamma_s)f_g(\tau)$. Since the exponential is surjective, we can choose an endomorphism $L_s \in \Endo_{\CC}(V)$ such that $\rho(\gamma_s) = e^{2\pi i L_s}$. We call $L_s$ a choice of exponents for $\rho(\gamma_s)$, or a choice of exponents at $s$. For every section of $\CC \to \CC/\ZZ$, there is a unique choice of $L_s$ with eigenvalues taking values in the section. If $\tilde f_g(\tau) \df e^{-2\pi i L_s\tau}f_g(\tau)$, then it follows that $\tilde f_g(\tau+w) = \tilde f_g(\tau)$. Since $\tilde f_g$ is holomorphic, it follows that $\tilde f_g$ admits a Fourier expansion
\begin{equation}
  \label{eq:qexp}
  \tilde f_g(q_w) = \sum_{n \in \ZZ} a_nq^n_w
\end{equation}
where $a_n \in \CC^d$ and $q_w = e^{2\pi i \tau/w}$. Whether this $q_w$-expansion is meromorphic is indepdent of the choice of $L_s$. One sees this by describing all possible choices of $L_s$, typically by using the Jordan canonical form of $\rho(\gamma_s)$, and comparing the results for the different choices.

A section of $\cV(\rho)$ in a punctured neighbourhood of the cusp $s$ extends to a section of $\cV_\dagger(\rho)$ at $s$ if and only if the $q_w$-expansion in \eqref{eq:qexp} is meromorphic. If $L = (L_s)_{s \in S}$ denotes a choice of exponents for each cusp, then let $\cV(\rho,L) \subseteq\cV_\dagger(\rho)$ denote the vector bundle on $X$ isomorphic to $\cV(\rho)$ on $Y$, and such that a section of $\cV_\dagger(\rho)$ at $s$ is contained in $\cV(\rho,L)$ if and only if the $q_w$-expansion \eqref{eq:qexp} obtained using the exponents $L$ is in fact holomorphic. Thus, for every choice of exponents $L$, we obtain inclusions of sheaves $\cV(\rho,L) \subseteq \cV_\dagger(\rho)$. These bundles are often called \emph{lattices} in the meromorphic extension $\cV_\dagger(\rho)$.

The meromorphic extension is naturally endowed with a connection
\[
  \nabla \colon \cV_\dagger(\rho) \to \cV_\dagger(\rho)\otimes \Omega^1_X(S) \cong \cV_\dagger(\rho) \otimes \Omega^1_X,
\]
where $\Omega^1_X(S)$ denotes the sheaf of regular differentials on $Y$ with at worst simple poles at the cusps. Viewing sections of $\cV_\dagger(\rho)$ as functions $f \colon \uhp \to V$, the holomorphic connection is described simply by $\nabla f =\frac{df}{d\tau}\otimes d\tau$. This connection keeps stable every lattice $\cV(\rho,L)$ and restricts to a regular connection
\[
  \nabla \colon \cV(\rho,L) \to \cV(\rho,L)\otimes \Omega^1_X(S)
\]
for every choice of exponents $L$.

The connection matrix at a cusp can be described explicitly as follows: for simplicity assume that the cusp is $\infty$ and it is of width $w > 0$ as above. For large enough $N$, the open set
\[
  U = \{x+iy \in \uhp \mid y > N,~ 0 < x < w\}
\]
satisfies $\gamma U \cap U = \emptyset$ for all $\gamma \in \Gamma$. Hence each vector $v \in V$ extends uniquely to a locally constant multivalued section $c_v$ of $\cV(\rho,L)$ over $j(U)$. By varying $v$ over a basis for $V$, we obtain a single valued frame
\[
  \tilde c_v(\tau) = e^{-2\pi i L_s\tau}c_v(\tau)
\]
for the fiber of $\cV(\rho,L)$ over the cusp $s$. Since $c_v$ is locally constant, it follows that
\[
  \nabla(\tilde c_v(\tau)) = -2\pi i L_s\tilde c_v(\tau).
\]
Thus, the connection matrix at the cusp $s$ is $-2\pi i L_s$. The \emph{residue} of a regular connection at a cusp $s$ is the connection matrix divided by $2\pi i$. Therefore,
\[\res_s(\cV(\rho,L),\nabla) = -L_s.\]

So far we have not incorporated the filtration $F$ of the filtered representation $(\rho,F)$ into our discussion of regular connections. Our next goal is to introduce a corresponding filtered object $(\cV(\rho,F),G)$, where for each $s \in S$ and real number $\alpha$, the bundle $G_{\alpha,s}\cV(\rho,F) \subseteq \cV_\dagger(\rho)$ is an extension of the bundle $\cV(\rho)$ to $X$. The stalk of sections of $G_{\alpha,s}\cV(\rho,F)$ over $s$ is generated by all sections of the following type: for every $v \in F_{\beta,s} V$, and for every choice of exponents $L_s$ satisfying $\rho(\gamma_s) = e^{2\pi i L_s}$ and such that the real parts $r$ of the eigenvalues of $L_s$ all satisfy $r <\beta-\alpha$, then $e^{-2\pi i L_s\tau}c_v$ defines a local section of $G_{\alpha,s}\cV(\rho,F)$ over $s$. 

This gives an example of the following type of filtered object.
\begin{dfn}
  A \emph{filtered bundle} on $X$ consists of a holomorphic bundle $E$ on $Y$, along with filtered families of extensions $E_{\beta,s}$ of $E$ to $s \in S$ for each cusp such that
  \begin{enumerate}
  \item (decreasing) $E_{\alpha,s} \subseteq E_{\beta,s}$ for all $\alpha \geq \beta$;
  \item (left continuous) for all $\beta \in \RR$, there exists $\veps > 0$ such that $E_{\beta-\veps,s} =E_{\beta,s}$;
  \item (periodicity) if $z$ is a local coordinate vanishing to order $1$ at $s$, then $E_{\beta+1,s} = zE_{\beta}$. 
  \end{enumerate}
\end{dfn}
If $E$ is a filtered bundle on $X$, then let $\bar E$ denote the vector bundle on $X$ obtained by using the extensions $E_{0,s}$ at each cusp. The fiber $\bar E(s)$ over $s$ is a finite dimensional vector space that inherits a decreasing filtration $\bar E_\alpha(s)$ from the extensions $E_{\alpha,s}$ for $\alpha \in [0,1)$. In more down to earth terms, when $E$ arsise as $\cV(\rho)$ above, $\bar E_\alpha(s)$ is spanned by constant terms at $s$ of modular forms that are sections of $E_{\alpha,s}$. The left continuous and periodicity conditions imply that $\bar E_{1-\veps}(s) = 0$ for some $\veps >0$.
\begin{dfn}
  The \emph{filtered degree} of a filtered bundle $E$ is the quantity
  \[
  \deg(E) \df  \deg(\bar E) + \sum_{s\in S}\sum_{0 \leq \alpha < 1} \alpha \dim(\Gr_\alpha (\bar E(s)))
\]
where $\deg(\bar E)$ is the standard notion of degree for the holomorphic bundle $\bar E$ on the compact curve $X$. The (filtered) \emph{slope} of $E$ is the quantity $\mu(E) = \deg(E)/\rank(E)$.
\end{dfn}

\begin{dfn}
A \emph{filtered regular connection} on $X$ consists of a holomorphic connection $(V,\nabla)$ on $Y$ endowed with the structure of a filtered bundle $V_{\beta,s}$ such that for every $s \in S$ and $\beta \in \RR$, such that the connection extends to a regular connection
\[
  \nabla \colon V_{\beta,s} \to V_{\beta,s} \otimes \Omega^1_X(s).
\]
The \emph{residue} $\res_s(V,\nabla)$ of $(V,\nabla)$ at $s$ is the residue of $(V_{0,s},\nabla)$ at $s$.
\end{dfn}

\begin{dfn}
  A filtered regular connection $(V,\nabla)$ is said to be \emph{semistable} provided that for every filtered regular subconnection $(U,\nabla) \subseteq (V,\nabla)$, the inequality $\mu(U,\nabla) \leq \mu(V,\nabla)$ holds. If the inequality holds without equality whenever $(U,\nabla)$ is nontrivial, then $(V,\nabla)$ is said to be \emph{stable}.
\end{dfn}

\begin{rmk}
Note that in the stability condition of a filtered regular condition, one only tests against filtered sub-objects, so that the filtered pieces must be preserved by the connection.
\end{rmk}

We now state, using the precise language and conventions adopted here, a classical theorem that forms one half of nonabelian Hodge theory.

\begin{thm}[Riemann-Hilbert correspondence]
  The association \[(\rho,F) \mapsto (\cV(\rho,F),G,\nabla)\]
  described above defines a degree-preserving equivalence between the category of filtered representations of $\Gamma$ and the category of filtered regular connections on $X$. This equivalence is compatible with direct sums, determinants, duals and tensor products. Semistable (resp. stable) objects correspond to one another under this equivalence. 
\end{thm}
\begin{proof}
A proof of all the statements save for the residue condition can be found in Lemma 3.2 of \cite{Simpson1}.
\end{proof}
\begin{rmk}
To lighten notation we will often write $\cV(\rho,F)$ for the filtered regular connection associated to a filtered representation $(\rho,F)$. Subscripts $\cV_{s,\beta}(\rho,F)$ indicate that this is a filtered collection of bundles that extend $\cV(\rho,F)$ to the cusp $s$. An overline $\cVbar(\rho,F)$ denotes the extension of $\cV(\rho,F)$ to $X$ using $\cV_{s,0}(\rho,F)$ at each cusp. The fiber of $\cVbar(\rho,F)$ over a cusp $s$ will be denoted $\cVbar(\rho,F,s)$. Its filtered pieces will be denoted as $\cVbar(\rho,F,s)_\beta$ to distinguish this filtration of a finite dimensional vector space from the filtration $\cV_{s,\beta}(\rho,F)$ by sheaves.
\end{rmk}

\begin{ex}
  Let $\Gamma = \PSL_2(\ZZ)$, for which the group of characters is cyclic of order $6$ and generated by $\chi \colon \Gamma \to \CC^\times$ satisfying $\chi(T) = e^{2\pi i/6}$, $\chi(S) = -1$, $\chi(R) = e^{4\pi i/3}$. For $a = 0,\ldots,5$, the choices of logarithm for $\chi^a$ are the rational numbers $\frac{a}{6}+n$ for $n \in \ZZ$. A filtered character consists of a pair $(\chi^a,b)$ where $b \in \RR$ is the jump of the filtration of $\chi^a$. The filtered degree $\deg(\chi^a,b)$ is equal to $b$. The residue is $\res_{\infty}(\chi^a,b) = e^{2\pi ia/6}$.

  Next we describe the corresponding filtered regular connection. The bundle on $Y$ underlying $\cV(\chi^a,b)$ is the usual bundle associated to $\chi^a$. For $\alpha \in \RR$, for every $\beta \leq b$ and for every choice of logarithm $\frac a6+n < \beta-\alpha$, if $c_v$ denotes the locally constant section associated to a basis $v \in V$ for the one-dimensional complex vector space $V$, then
  \[
  e^{-2\pi i(\frac a6+n)\tau}c_v(\tau) =q^{-n}e^{-2\pi i\tau/6}c_v(\tau) 
\]
defines a local section of $\cV_{\alpha,\infty}(\chi^a,b)$. Therefore, if $\cV(\chi^a)$ denotes the canonical extension of $\chi^a$ defined by the logarithm $a/6$, then sections of $\cV_{\alpha,\infty}(\chi^a,b)$ are sections of $\cV(\chi^a)$ with a pole of order $n < b-\frac a6-\alpha$ at the cusp. It follows that
\[
  \cV_{\alpha,\infty}(\chi^a,b) \cong \cV(\chi^a,a/6) \otimes \cO_X\left(\floor{b - \frac a6 -\alpha}\infty\right) 
\]
where $\cO_X(m\infty)$ denotes the sheaf of holomorphic functions on $Y$ with a pole of order at most $m$ at the cusp (meaning that the functions vanish at the cusp if $m < 0$), and $\cV(\chi^a,a/6)$ denotes the canonical extension of $\cV(\chi^a)$ to the cusp using the exponent $a/6$. The jumps occur at the values of $\alpha \in  b-\frac{a}{6}+\ZZ$. Therefore, the jump in $[0,1)$ occurs at $\{b-(a/6)\}$. To compute the filtered degree of $\cV(\chi^a,b)$ we now need to compute the \emph{unfiltered} degree of $\deg \cV_{0,\infty}(\chi^a,b)$:
\begin{align*}
  \deg \cV_{0,\infty}(\chi^a,b) &= \deg\cV(\chi^a,a/6) \otimes \cO_X\left(\floor{b - \frac a6}\infty\right)\\
                                &= (a/6)+\floor{b - (a/6)}\\
                                &= (a/6)+(b - (a/6)-\{b-(a/6)\})\\
  &= b-\{b-(a/6)\}.
\end{align*}
This shows that the filtered degree of $\cV(\chi^a,b)$ is also $b$.
\end{ex}

\subsection{Filtered Higgs bundles}
Filtered Higgs bundles provide a third collection of filtered objects, whose relationship to filtered connections brings about the other half of nonabelian Hodge theory. While the connection between filtered local systems and filtered regular connections is relatively classical and has long been a part of the study of modular forms, the study of modular forms through Higgs bundles is a more recent development. 
\begin{dfn}
A \emph{filtered regular Higgs bundle} on $X$ consists of a pair $(E,\theta)$ where $E$ is a filtered bundle on $X$, and $\theta$ is a collection of $\cO_X$-linear maps
\[
  \theta \colon E_{\beta,s} \to E_{\beta,s} \otimes \Omega^1_X(s).
\]
The degree of a filtered regular Higgs bundles is the degree of the underlying filtered bundle on $X$. The \emph{residue} of $(E,\theta)$ at $s \in S$ is the operator $\theta\left(q\frac{d}{dq}\right)$ acting on the fiber of $E_{0,s}$ over $s$, where $q$ is a local parameter at $s$.
\end{dfn}
\begin{dfn}
  A filtered regular Higgs bundle $(E,\theta)$ is said to be \emph{semistable} provided that for every filtered subbundle $E'\subseteq E$ satisfying $\theta(E') \subseteq E'$, one has $\mu(E',\theta) \leq \mu(E,\theta)$. If the inequaltiy is strict whenever $(E',\theta)$ is a nontrivial subobject, then $(E,\theta)$ is said to be \emph{stable}.
\end{dfn}

As will be discussed below, nonabelian Hodge theory extends the Riemann-Hilbert correspondence to a triangle of equivalences including the category of filtered Higgs bundles.

\section{Vector valued modular forms}
\subsection{Definition} For all even integers $k$, let $\cL_k$ denote the line bundle on $X$ that pulls back to $\uhp$ as the trivial bundle $j^*\cL_k = \uhp \times \CC$ endowed with the action $\gamma(\tau,v) = (\tau, (c\tau+d)^kv)$ for $\gamma = \stwomat abcd \in \Gamma$. The extension of this bundle to the cusps is the canonical one, so that global sections of $\cL_k$ are the usual scalar valued holomorphic modular forms for $\Gamma$ of weight $k$.

Let $\rho \colon \Gamma \to \GL(V)$ denote a complex finite dimensional representation of a Fuchsian group $\Gamma \subseteq \PSL_2(\ZZ)$. A weakly holomorphic modular form for $\rho$ of weight $k$ is a global holomorphic section of the bundle $\cV_\dagger(\rho)\otimes \cL_k$. That is, such an object is a holomorphic function $f \colon \uhp \to V$ such that
\[
  f(\gamma \tau) = (c\tau+d)^k\rho(\gamma)f(\tau)
\]
for all $\gamma =\stwomat abcd \in \Gamma$, and which is \emph{meromorphic at the cusps} in the following sense: for each cusp $s$, choose $\gamma_s \in \Gamma$ stabilizing $s$ such that $g\gamma_sg^{-1} = \stwomat 1w01$ for some $w > 0$ and $g \in \PSL_2(\RR)$. Then for any choice of logarithm $L_s$ matrix satisfying $\rho(\gamma_s) = e^{2\pi i L_s}$, the function $e^{-2\pi i L_s\tau}f(\tau)$ has a meromorphic $q_r = e^{2\pi i\tau/r}$ expansion.

One often wants to impose bounds on the poles at cusps (or even holomorphy or vanishing), and this is achieved through the mechanism of fixing exponent matrices $L_s$ at each cusp. This defines a holomorphic bundle $\cV(\rho,L)$ on the compact curve $X$ as discussed in the previous section. A global section $f$ of $\cV_{k}(\rho,L) \df\cV(\rho,L)\otimes \cL_k$ is a section of $\cV_\dagger(\rho)\otimes \cL_k$ such that the function $e^{-2\pi i L_s\tau}f(\tau)$ has a \emph{holomorphic} $q_r = e^{2\pi i\tau/r}$ expansion at each cusp. We write
\[
  M_{k}(\rho,L) \df H^0(X,\cV_k(\rho,L))
\]
for the corresponding (finite dimensional) space of modular forms. The usual holomorphic modular forms correspond to choosing logarithms $L_s$ with real parts of their eigenvalues inside $[0,1)$, the \emph{canonical choice} of exponents. Cusp forms correspond to the interval $(0,1]$, and in general it is advantageous to allow arbitrary exponents. We will see some justification for this below, but one might also mention functorial linear algebraic constructions such as tensor products, symmetric powers, induction, et cetera. These functorial constructions lead naturally to non-canonical exponents --- see \cite{FrancMason2} for examples.

\subsection{The three-term inequality}
The authors began to think about possible applications of nonabelian Hodge theory to vector valued modular forms when discussing an open question concerning the case of $\Gamma = \PSL_2(\ZZ)$. In this case $X$ can be identified with the weighted projective line $\PP(2,3)$. Every line bundle on $\PP(2,3)$ is isomorphic with $\cO(k)$ for an even integer $k$\footnote{Since we are working with $\PSL_2(\ZZ)$ as opposed to $\SL_2(\ZZ)$, all weights should be even. Hence $\cO(1)$ is not defined on $X$, but rather, it is a line bundle on a $\mu_2$-gerbe over $X$. We will only work with even weights in this paper.}. For example, one has $\cL_k \cong \cO(k)$. Moreover every holomorphic bundle on $\PP(2,3)$ decomposes into a direct sum of line bundles, and this decomposition is unique up to isomorphism. For details on these facts one can consult \cite{CandeloriFranc}.

In this simplified case a choice of logarithms for a representation $\rho$ is a single matrix $L$ satisfying $\rho\stwomat 1101 = e^{2\pi i L}$, and the corresponding module
\[
  M(\rho,L) = \bigoplus_{k \in \ZZ} M_{k}(\rho,L)
\]
of modular forms is known to be free of rank $d = \dim \rho$ over the ring $\CC[E_4,E_6] \cong \bigoplus_{k \in \ZZ} H^0(X,\cL_k)$ of holomorphic scalar valued modular forms of level one. If $d = \dim \rho$, one can always choose a basis of forms of weights $k_1,k_2,\ldots, k_d$ for $M(\rho,L)$ over $\CC[E_4,E_6]$, where
\[
  \cV_0(\rho,L) \cong \bigoplus_{j=1}^d \cL_{-k_j}.
\]
Let $m_k$ denote the multiplicity of $\cL_{-k}$ occuring in the decomposition above. In \cite{FrancMason} it was observed that these weight multiplicities satisfy the inequalities
\begin{equation}
\label{eq:threeterm}
  m_k \leq m_{k+2}+m_{k-2}
\end{equation}
for irreducible representations $\rho$ of dimension at most $12$. Moreover in the case that $\rho$ is irreducible and \emph{unitary} of arbitrarily large dimension, \cite{FrancMason} proved that this three-term inequality holds. We expect this inequality to hold for all irreducible representations regardless of unitarity hypotheses. Since the nonabelian Hodge correspondence arose from considerations of extending classical results of Narasimhan-Seshadri \cite{NarasimhanSeshadri} (in the compact case) and Mehta-Seshadri \cite{MehtaSeshadri} (in the noncompact case) from unitary bundles to all holomorphic bundles, it is natural to suggest that the study of modular forms associated to nonunitary representations could profit from contact with nonabelian Hodge theory. In particular, can the three-term inequality \eqref{eq:threeterm} be verified for all irreducible representations of $\PSL_2(\ZZ)$ by exploiting the rich geometry of the moduli space of filtered Higgs bundles by, say, deforming to the unitary case or to complex variations of Hodge structure? We begin a discussion of this question here and provide some partial answers.

\section{Nonabelian Hodge theory}
\subsection{The nonabelian Hodge correspondence}
Here, we review the parts of the full nonabelian Hodge theorem that we need for the current investigation.  For a broader survey on the nonabelian Hodge correspondence, we refer the reader to either of \cite{DonagiPantev,RabosoRayan}.  To begin, we let $\rho \colon \Gamma \to \GL(V)$ denote a representation of a Fuchsian group $\Gamma$. The nonabelian Hodge correspondence is mediated by the existence of a Hermitian metric on $\cV(\rho)$ satisfying a nonlinear differential equation, and with prescribed growth conditions at cusps. In order to describe this in concrete automorphic terms, let $\Herm_n$ denote the space of $n\times n$ positive definite Hermitian matrices, and let $K \colon \uhp \to \Herm_n$ denote a smooth map. Define a Hermitian metric on $\cV(\rho)$ by setting
\[
  \langle u,v\rangle_\tau \df u^T\cdot K(\tau) \cdot \bar v.
\]
In order for this to descend to a well-defined metric on $\cV(\rho)$ over the open curve $Y = \Gamma \backslash \uhp$, it is necessary and sufficient that $\langle u,v\rangle_{\tau} = \langle \rho(\gamma)u,\rho(\gamma)v\rangle_{\gamma \tau}$ for all $\gamma \in \Gamma$. That is, $K$ must satisfy the transformation law
\[
  K(\tau) = \rho(\gamma)^T\cdot K(\gamma \tau) \cdot \overline{\rho(\gamma)}
\]
for all $\tau \in \uhp$ and $\gamma \in \Gamma$. This means that $K$ is a matrix valued smooth modular form of weight zero taking values in $\Herm_n$. We have not yet imposed any conditions of slow growth at the cusps.

Let $\cA^1_Y$ denote the sheaf of smooth differential one forms on $Y$. Writing also $\cV(\rho)$ for the smooth sheaf associated to the vector bundle $\cV(\rho)$, then coordinatewise exterior differentiation defines a map of sheaves
\[
  d\colon \cV(\rho) \to \cV(\rho)\otimes \cA^1_Y.
\]
In fact, $d$ is a flat connection since the exterior derivative is flat. The decomposition $\cA^1_{Y} = \cA^{1,0}_{Y}\oplus \cA^{0,1}_Y$ yields a decomposition $d = \partial + \bpartial$ into holomorphic and antiholomorphic components
\begin{align*}
\partial &\colon \cV(\rho) \to \cV(\rho)\otimes \cA^{1,0}_Y, & \bpartial &\colon \cV(\rho) \to \cV(\rho)\otimes \cA^{0,1}_Y. 
\end{align*}
To be explicit, if a section $f$ of $\cV(\rho)$ has coordinate functions $f_j$, then
\begin{align*}
  \partial \left(\begin{matrix}
      f_1(\tau)\\
      f_2(\tau)\\
      \vdots\\
      f_n(\tau)
    \end{matrix}\right)
& = \left(\begin{matrix}
    (\partial f_1)(\tau)\\
    (\partial f_2)(\tau)\\
    \vdots\\
    (\partial f_n)(\tau)
  \end{matrix}\right)d\tau, &
                              \bpartial \left(\begin{matrix}
      f_1(\tau)\\
      f_2(\tau)\\
      \vdots\\
      f_n(\tau)
    \end{matrix}\right)
& = \left(\begin{matrix}
    (\bpartial f_1)(\tau)\\
    (\bpartial f_2)(\tau)\\
    \vdots\\
    (\bpartial f_n)(\tau)
    \end{matrix}\right)d\btau. 
\end{align*}
The bundle $\cV(\rho)$ has a natural holomorphic structure given by $\bpartial$ such that the holomorphic sections for this structure are holomorphic functions $f\colon \uhp \to \CC^n$ in the usual sense.

If $K$ is a Hermitian metric on $\cV(\rho)$, then a smooth connection $D$ on $\cV(\rho)$ is said to be a \emph{metric connection} provided that
\[
  d(K(u,v)) = K(Du,v) +K(u,Dv)
\]
for all sections $u$, $v$ of $\cV(\rho)$. The constructions of nonabelian Hodge theory rely on choosing a $(1,0)$-connection $\delta_K'$ and a $(0,1)$-connection $\delta_K''$ such that $\partial + \delta_K''$ and $\bar\partial + \delta_K'$ are both metric connections relative to $K$. On a curve these operators can be expressed in terms of matrices as $\delta_K' = \partial + Md\tau$ and $\delta_K'' = \bpartial + Nd\btau$, where $M$ and $N$ are two matrices of smooth functions on $Y$. By examining the $(1,0)$-component, or equivalently, the $(0,1)$-component, of the metric connection condition for each of these derivations, one finds the identities:
\begin{align*}
\delta_K' &= \partial + \partial\log \bar K d\tau,\\
\delta_K'' &= \bpartial + \bpartial\log \bar Kd\btau.
\end{align*}
Above, if $G$ is an invertible matrix of functions and $D$ is a derivation, we write $D\log G \df G^{-1}D(G)$ where $D(G)$ denotes the matrix obtained from $G$ by applying $D$ entrywise. 

Using these operators, we can describe the rest of the constructions of nonabelian Hodge theory. They involve the following list of operators:
\begin{align*}
\partial_K &= \partial + \frac 12 \partial \log \bar K d\tau, & \bpartial_K &= \bpartial + \frac 12 \bpartial \log \bar K d\btau,\\
\theta_K &= -\frac 12\partial \log \bar K d\tau, &\bar\theta_K &= -\frac 12 \bpartial \log \bar K d\btau.
\end{align*}
For a general choice of $K$, these operators do not need to have desirable properties.  In particular, the operators $D_K' = \partial_K + \bar\theta_K$ and $D_K'' = \bar\partial_K + \theta_K$ will not be flat in general. But it turns out that if at least one of them is flat, then both are, and they allow one to pass back and forth between flat regular connections and regular Higgs bundles.

Recall that a Higgs bundle on $Y$ consists of a triple $(E, \delta,\theta)$ where $E$ is a smooth bundle, $\delta \colon E \to E \otimes \cA^{0,1}_Y$ is a holomorphic structure on $E$, and $\theta \colon E \to E \otimes \Omega^1_Y$ is holomorphic for the complex structure defined by $\delta$. If $E$ is the smooth bundle underlying $\cV(\rho)$, whenever $(D_K'')^2 = 0$, then $K$ is called a \emph{harmonic metric} for $\rho$, and the triple $(E,\bpartial_K,\theta_K)$ is the Higgs bundle associated to $\rho$ via nonabelian Hodge theory. While the two bundles associated to one another via nonabelian Hodge theory have the same underlying smooth bundle, in general their complex structures will be different. They will be equal only when $\partial \log \bar K = 0$, as is clear from the equations above.

The succinct condition $(D_K'')^2 = 0$ can be made more explicit. Since $Y$ is a curve, a straightforward computation shows that this condition simplifies to 
\[
  \bpartial\partial \log \bar K = \frac 12 [\partial \log \bar K,\bpartial \log \bar K]
\]
or equivalently $\partial\bpartial \log K = \frac 12[\bpartial \log K,\partial \log K]$. Observe that if $K$ is a metric on a line bundle, then the commutator vanishes and this simply says that $K$ is log harmonic in the sense of Definition 4.1.7 of \cite{GoldmanXia}. We thus obtain the following automorphic description of a harmonic metric on a uniformized flat regular connection.
\begin{dfn}
Let $\rho \colon \Gamma \to \GL_d(\CC)$ denote a representation of a Fuchsian group $\Gamma$. Then a \emph{harmonic metric}  for $\rho$ is a smooth function $K \colon \uhp \to \Herm_d$ valued in positive definite Hermitian matrices, and such that the following two conditions are satisfied:
\begin{enumerate}
\item for all $\gamma \in \Gamma$ one has
\[
  K(\gamma\tau) = \rho(\gamma)^{-T}\cdot K(\tau) \cdot \overline{\rho(\gamma)}^{-1};
\]
\item the harmonic equation is satisfied:
\[
\partial \bpartial \log(K) = \frac 12[\bpartial \log(K),\partial \log(K)].
\]
\end{enumerate}
\end{dfn}

Note that we have not yet imposed any conditions at the cusps.

\begin{ex}
  The simplest examples are unitary representations $\rho$. In this case the constant identity map $K(\tau) = I$ is a harmonic metric, the complex structure is the standard one under a uniformization, and the corresponding Higgs field is the zero endomorphism. This case gives rise to the usual Petersson scalar product in the theory of modular forms. Thus, one can regard harmonic metrics as generalizations of the Petersson product to nonunitary representations.

  More generally, if $\Gamma$ has cusps and $\rho$ is endowed with filtrations at the cusps, then growth conditions must be imposed at the cusps. Even for unitary representations $\rho$, the harmonic metric with appropriate growth conditions will differ from the Petersson scalar product if the filtration is nontrivial. We will discuss this below.
\end{ex}

\begin{ex}
If $K$ is a harmonic metric for $\rho$, and if $\chi$ is a unitary one-dimensional representation of $\Gamma$, then $K$ is also a harmonic metric for $\rho \otimes \chi$.
\end{ex}

\begin{ex}
  \label{ex:totallygeodesic}
A more interesting class of examples is provided by the \emph{totally geodesic} metrics. Such examples arise as follows: a metric $K \colon \uhp \to \Herm_n$ is harmonic if and only if the trace of the second fundamental form of $K$ vanishes (see Theorem 3.3.3 of \cite{BairdWood} for details). If the second fundamental form itself vanishes, then $K$ is a totally geodesic map. This means that any geodesic in $\Herm_n$ tangent to the image of $K$ at some point is in fact entirely contained in the image of $K$. By Proposition I.4.12 of \cite{BorelJi}, all such (irreducible) totally geodesic metrics $K$ arise from injective irreducible representations $\alpha \colon \SL_2(\RR) \to \GL_n(\CC)$ making the following diagram commute:
\[
\xymatrix{
\SL_2(\RR)\ar[r]^{\alpha}\ar[d]_{a \mapsto ai} & \GL_n(\CC)\ar[d]^{g \mapsto g\overline g^T}\\
\uhp \ar[r]_K& \Herm_n
} 
\]

This construction can be used to obtain the harmonic metric for the inclusion representation $\rho \colon \Gamma \to \GL_2(\CC)$ of a Fuchsian subgroup $\Gamma$ of $\SL_2(\RR)$. Let $\rho^* \colon \SL_2(\RR) \to \GL_n(\CC)$ denote the dual of the inclusion, so that $\rho^*(a) = a^{-T}$. The resulting totally geodesic metric $K$ can be described, as usual, by lifting $\tau = x+iy \in \uhp$ to the matrix $a_\tau=\stwomat{y^{1/2}}{xy^{-1/2}}{0}{y^{-1/2}} \in \SL_2(\RR)$ so that
\[
  K(\tau) = \rho^*(a_\tau)\overline{\rho^*(a_\tau)}^T = \frac{1}{y}\twomat{1}{-x}{-x}{x^2+y^2}.
\]
This is the harmonic metric for the inclusion representation of a Fuchsian subgroup of $\SL_2(\RR)$. The Higgs bundle corresponding to the inclusion representation under the nonabelian Hodge correspondence thus has a complex structure and corresponding Higgs field defined by the operators:
\begin{align*}
  \bpartial_K &= \bpartial + \frac{1}{(\tau-\btau)^2}\twomat{\tau}{-\tau^2}{1}{-\tau}d\btau, & \theta_K &= \frac{1}{(\tau-\btau)^2}\twomat{-\btau}{\btau^2}{-1}{\btau}d\tau.
\end{align*}
Observe that the inclusion representation is not unitarizable, $\bpartial_K$ does not yield the standard complex structure on the uniformized bundle, and the Higgs field $\theta_K$ is nonzero.

Totally geodesic metrics have appeared in the literature on automorphic forms -- see for example \cite{Harris1}, \cite{Harris2}, \cite{Harris3}, \cite{EsnaultHarris}.
\end{ex}

\begin{rmk}
Since we have insisted that Fuchsian groups are contained in $\PSL_2(\RR)$, in the previous Example \ref{ex:totallygeodesic} we should have twisted the inclusion of the subgroup  $\Gamma \subseteq \SL_2(\RR)$ by a unitary character to be consistent with this hypothesis. However, the harmonic metric is the same in both cases, and so in that example we worked inside $\SL_2(\RR)$ to avoid introducing an unnecessary twist that would have complicated matters without adding clarity. More generally, the results of Simpson and this paper extend in a straightforward way to subgroups of $\SL_2(\RR)$ (or higher degree covers) and their associated gerbes over $\uhp$ by twisting with a unitary character to descend to $\PSL_2(\RR)$.
\end{rmk}

\begin{rmk}
  In general it is a difficult problem to write down exact formulas for harmonic metrics other than in the rank one case, or in the higher rank unitary or totally geodesic cases discussed above. The existence proofs discussed in \cite{Hitchin}, \cite{Donaldson}, \cite{Corlette}, \cite{Simpson1}, \cite{Simpson2}, \cite{Biquard}, and \cite{BiquardBoalch} are not constructive in nature. A natural question is to give a representation-theoretic proof of the existence of harmonic metrics that utilizes the representation theory of (possibly infinite-dimensional) Lie groups and algebras, or even more exotic structures, such as vertex operator algebras (VOAs).  Indeed, connections between Higgs bundles and VOAs have already emerged in a number of different contexts, especially in regards to the quantization of the Hitchin system, that is, of the natural integrable system structure on the mdouli space of ordinary Higgs bundles on a smooth compact curve.  See \cite{DonagiKatzSharpe,BenZviFrenkel} for instance.
  
  Such a proof concerning harmonic metrics would make the relationship between automorphic forms and nonabelian Hodge theory more transparent. For example, what is the form of the Selberg trace formula for the $L^2$-space associated to a tame harmonic metric for a nonunitary representation? In the unitary case this goes back to Selberg and Hejhal (see \cite{Hejhal1}, \cite{Hejhal2}), and more recently Muller \cite{Muller} has established a version of the Selberg trace formula for a non-selfadjoint flat Laplacian associated to nonunitary representations, in the case of cocompact groups $\Gamma$.  (See also \cite{Deitmar}.)  What more can one say if one works systematically with a harmonic metric?  Hopefully, these questions will motivate further work and interesting results down the line.

\end{rmk}

\begin{dfn}
  \label{d:tamemetric}
  Let $(\rho,F)$ denote a filtered representation of $\Gamma$. Then a \emph{tame harmonic metric} for $(\rho,F)$ is a harmonic metric $K$ for $\rho$ satisfying the following tameness condition. For every cusp $s \in S$, for every $\beta \in \RR$ and $\veps > 0$, the following holds: there exists $N >0$ and constants $c$ and $C$ such that for all $v \in F_{\beta,s}V$ but $v \not \in \bigcup_{\alpha > \beta}F_{\alpha,s}V$, if $\im(\tau) > N$ then
  \[
  c\abs{q_w}^{\beta+\veps} \leq  \abs{v}_K \leq C\abs{q_w}^{\beta-\veps}.
\]

A harmonic metric for $\rho$ is said to be \emph{tame} if  it is tame with respect to some filtered representation $(\rho,F)$.
\end{dfn}

Recall that in Definition \ref{d:tamemetric}, $q_w = e^{2\pi i \tau/w}$ is a local coordinate at the cusp $s$. The width $w> 0$ of the cusp $s$ is characterized by the fact that the stabilizer $\gamma_s\in \Gamma$ for $s$ is conjugate to $\stwomat 1w01$. The length $\abs{v}_K$ is a smooth function of $\tau \in \uhp$ equal to $\sqrt{v^TK(\tau)\bar v}$.

\begin{ex}
  Suppose that the filtration $F$ is \emph{trivial}, meaning that for each cusp $s$,
\[
  F_{\beta,s} V = \begin{cases}
    V & \beta \leq 0,\\
    0 & \beta > 0.
  \end{cases}
\]
In this case we can apply the condition of Definition \ref{d:tamemetric} to all basis vectors, with $\beta = 0$. If $K_{ij}(\tau)$ is the $(i,j)$-entry of $K$, then tameness is seen to be equivalent to the inequality
\[
  c\abs{q_w}^\veps \leq K_{ij}(g\tau) \leq C\abs{q_w}^{-\veps}
\]
for all $\tau$ with $\im(g\tau)$ large enough. So for example, if the inclusion representation is endowed with the trivial filtration, then the totally geodesic metric from Example \ref{ex:totallygeodesic} is seen to be tame for this filtration.
\end{ex}

In \cite{Simpson1}, Simpson usually expresses tameness of a harmonic metric in terms of the associated filtered regular Higgs bundle. Above we have expressed it in terms of the filtered local system, as this is the object that is most closely tied to the classical theory of modular forms. One of the key results in \cite{Simpson1} is the equivalence of these definitions of tameness. Tameness of a harmonic metric can be described intrinsically without reference to a filtered local system, a filtered regular connection, or a filtered regular Higgs bundle. Simpson introduced the notion of a \emph{harmonic bundle} to denote such an object.

\begin{dfn}
A \emph{harmonic bundle} consists of a representation $\rho \colon \Gamma \to \GL(V)$, and an associated harmonic metric $K$ that is \emph{tame} in the following sense: the metric grows at most polynomially in the Euclidean distance to each puncture, where the Euclidean distance is measured as usual in terms of a local coodinate $q_w$ at each cusp.
\end{dfn}

Tame harmonic bundles form a category, and they determine regular filtrations on the associated representations, connections and Higgs bundles. A harmonic bundle $(\rho,K)$ is \emph{irreducible} if it is impossible to decompose $\rho$ and $K$ compatibly into a nontrivial direct sum. A filtered object on $X$ will be said to be \emph{polystable} if it is a direct sum of stable filtered objects. The following Theorem is the main result of \cite{Simpson1}.
\begin{thm}
  Let $\Gamma \subseteq \PSL_2(\RR)$ be a Fuchsian group, $Y = \Gamma \backslash \uhp$ and let $X = Y \cup S$ be the compactification by cusps. The constructions discussed above define equivalences of categories between all four of the following categories:
  \begin{enumerate}
  \item tame harmonic bundles on $X$;
  \item polystable filtered representations of $\Gamma$;
  \item polystable filtered regular connections on $X$;
  \item polystable filtered regular Higgs bundles on $X$.
  \end{enumerate}
  Irreducible harmonic bundles correspond to irreducible filtered representations, stable filtered regular connections and stable filtered regular Higgs bundles.
\end{thm}
\begin{proof}
See page 755 of \cite{Simpson1}. Full-faithfullness and essential surjectivity of the various functors are established in Sections 4 and 6, respectively, of \cite{Simpson1}. Since the proofs take place locally on the universal cover $\uhp$ and at the cusps, they extend essentially without change to the case where $\Gamma$ contains elliptic points.
\end{proof}

Suppose that $(\rho,F)$ is a filtered representation with corresponding tame harmonic metric $K$. Let $(\cV,\nabla)$ and $(E,\theta)$ denote the corresponding filtered regular flat connection and filtered regular Higgs bundle, respectively. For each $s \in S$ we have defined residues for each of these filtered objects. In \cite{Simpson1}, Simpson describes how these residues are related to one another. The description involves breaking up the operator according to Jordan blocks. Then the nonabelian Hodge correspondence respects this Jordan decomposition, but it shifts both eigenvalues and jumps in a prescribed manner. The correspondence is as in the following table.

\begin{center}
  \begin{tabular}{|c|c|c|c|}
    \hline
    &$(\rho,F)$& $(\cV,\nabla)$& $(E,\theta)$\\
    \hline
    Jump &$\beta$ &$\beta+u$&$-u$\\
    \hline
    Eigenvalue &$e^{2\pi i(u+vi)}$&$-(u+vi)$&$-\frac 12(\beta+vi)$\\
    \hline
  \end{tabular}
\end{center}

Note that we have a slightly different labelling of variables than that found in \cite{Simpson1}. Ours preferences the filtered representation, while that of \cite{Simpson1} preferences the Higgs bundle. One surprising feature worth noting is that the jumps in the filtration of the representation shows up in the eigenvalues of the Higgs field. Thus, even if $\rho$ is a unitary representation, the Higgs field will be nonzero if the jumps of $(\rho,F)$ are nonzero.

\subsection{Higgs forms}
Fix now a representation $\rho$ of a Fuchsian group $\Gamma$ and a tame harmonic metric $K$ for $\rho$. The Higgs or Dolbeault side of nonabelian Hodge theory has the following automorphic interpretation: for an integer $k$, let $\cA_k(\rho)$ denote the space of smooth automorphic forms for $\rho$ of weight $k$, where we do not impose, initially, any conditions at the cusps. Such an automorphic form is a smooth function $f\colon \uhp \to \CC^n$ satisfying
\[
  f(\gamma\tau) = (c\tau+d)^k\rho(\gamma) f(\tau)
\]
for all $\gamma = \stwomat abcd \in \Gamma$. Define the subset $\Higgs_{\dagger,k}(\rho,K)\subseteq \cA_k(\rho)$ to be those forms $f \in \cA_k(\rho)$ satisfying the linear differential equation $\bpartial_Kf = 0$ and which have slow growth (relative to $K$) at the cusps. We call such an automorphic form a \emph{Higgs form} of weight $k$ for $\rho$.

Observe that $\Higgs_{\dagger,0}(\rho,K)$ is nothing but the space of global holomorphic sections of $(E,\theta) = (\cV,\bpartial_K,\theta_K)$. By abuse of notation let $\theta_K$ denote the matrix of functions appearing in the Higgs field. Then this matrix defines a linear map 
\[
  \theta_K \colon \Higgs_{\dagger,k}(\rho,K) \to \Higgs_{\dagger,k+2}(\rho,K).
\]
If we write
\[
  \Higgs_\dagger(\rho,K) = \bigoplus_{k \in \ZZ} \Higgs_{\dagger,k}(\rho,K)
\]
then this becomes a module over the ring of holomorphic scalar valued modular forms $M(\Gamma)$, and $\theta_K$ defines a graded $M(\Gamma)$-linear endomorphism of $\Higgs_\dagger(\rho,K)$. We stress that $\theta_K$ is \emph{linear}, and not a derivation.

\begin{rmk}
  In fact, since we have only imposed slow growth at cusps, $\Higgs_\dagger(\rho,K)$ even has the structure of an $M(\Gamma)[1/\Delta]$-module, where $\Delta$ is the Ramanujan $\Delta$ function, and $\theta_K$ also commutes with $\Delta^{-1}$.
\end{rmk}

\begin{ex}
Consider again Example \ref{ex:totallygeodesic}. The harmonic metric described there is the tame harmonic metric for the inclusion representation endowed with the trivial filtration. The differential equation satisfied by Higgs forms $f \in \Higgs_k(\rho,K)$ is
\[
  \bpartial(f) =\frac{1}{(\tau-\btau)^2}\twomat{-\tau}{\tau^2}{-1}{\tau}f.
\]
If we write $f=  (f_j)$ in coordinates, then this is the pair of equations
\begin{align*}
\bpartial(f_1) &= \frac{1}{(\tau-\btau)^2}(-\tau f_1+\tau^2f_2),\\
\bpartial(f_2) &= \frac{1}{(\tau-\btau)^2}(-f_1+\tau f_2).
\end{align*}
Observe that then $\bpartial(f_1-\tau f_2) = 0$, so that $g = f_1-\tau f_2$ is holomorphic in the usual sense, and $\bpartial(f_2) = -\frac{1}{(\tau-\btau)^{2}}g$. It follows that there is another holomorphic function $h$ so that $f_2(\tau) = h(\tau)-\frac{1}{\tau-\btau}g(\tau)$ and
\[
  f(\tau) = \twomat{-\frac{\btau}{\tau-\btau}}{\tau}{-\frac{1}{\tau-\btau}}{1}\twovec{g(\tau)}{h(\tau)}.
\]
The automorphy of $f$ implies that $g$ and $h$ are classical holomorphic scalar valued modular forms for $\Gamma$ of weights $k-1$ and $k+1$, respectively. Hence left multiplication by $M(\tau) = \stwomat{-\frac{\btau}{\tau-\btau}}{\tau}{-\frac{1}{\tau-\btau}}{1}$ defines an isomorphism
\[
  M(\tau) \colon  M_{\dagger,k-1}(\Gamma)\oplus M_{\dagger,k+1}(\Gamma) \stackrel{\cong}{\longrightarrow} \Higgs_{\dagger,k}(\rho,K).
\]

The matrix of $\theta_K$ in the basis given by $g$ and $h$ is the nilpotent matrix
\[M(\tau)^{-1}\theta_KM(\tau) = \twomat 0100.\]
For each $\lambda \in \CC^\times$ there is an isomorphism $a_\lambda \colon(\cV,\bpartial_K,\theta_K) \cong (\cV,\bpartial_K,\lambda \theta_K)$, and hence this Higgs bundle is fixed by the natural $\CC^\times$-action on Higgs bundles, which rescales the Higgs field. Indeed, the isomorphism is
\[
  a_\lambda = M(\tau)\twomat {\lambda}{0}{0}{1} M(\tau)^{-1} = \frac{1}{\tau-\btau}\twomat{\tau-\lambda\btau}{(\lambda-1)\tau\btau}{1-\lambda}{\lambda\tau-\btau}.
\]
\end{ex}

Since the metric $K$ is tame, it determines filtrations on all of the objects involved in the nonabelian Hodge correspondence. If $(E,\theta_K)$ is the filtered regular Higgs bundle corresponding to $(\rho,K)$, then recall that we write $\Ebar$ for the bundle on $X$ obtained by extending $E$ to each cusp using the zeroth filtered piece of $E$. Define
\[
  \Higgs_k(\rho,K) = H^0(X,\Ebar \otimes \cL_k).
\]
Since $\Ebar$ is a holomorphic bundle on a compact space, each space $\Higgs_k(\rho,K)$ is a finite dimensional vector space. The graded vector space
\[
  \Higgs(\rho,K) = \bigoplus_{k \in \ZZ} \Higgs_k(\rho,K)
\]
has the structure of a graded $M(\Gamma)$-module. As above, the Higgs field defines an $M(\Gamma)$-linear graded endomorphism of this module that increases weights by two. The reader should beware, however, that the Higgs field is the zero map in most cases of classical interest (for example, when $\rho$ is unitary and the filtration induced on $\rho$ by $K$ is trivial).
\begin{thm}
Let $\Gamma = \PSL_2(\ZZ)$, and let $Y$ and $X$ denote the corresponding open and compact curves, respectively. If $(\rho,K)$ describes an irreducible tame harmonic bundle on $X$, then the module $\Higgs(\rho,K)$ of Higgs forms is free of rank $\dim \rho$ over the ring $M(\Gamma) = \CC[E_4,E_6]$ of scalar valued holomorphic modular forms for $\rho$.
\end{thm}
\begin{proof}
The proof is the same as the geometric proof of the free module theorem for vector valued modular forms from \cite{CandeloriFranc}. It uses the splitting principle for vector bundles on $X=\PP(2,3)$, and the fact that the spaces $\Higgs_k(\rho,K)$ are the global sections of the twisted bundles $\Ebar \otimes \cL_k$ on $X$. Since $\Pic(X) \cong \ZZ$ with $\cL_1\in \Pic(X)$ a generator, and $M(\Gamma) =\bigoplus_{k \in \ZZ} H^0(X,\cL_k)$, the proof follows.
\end{proof}

Now, given a tame harmonic bundle $(\rho,K)$, we obtain two spaces of modular forms: the more classical space $M(\rho,F)$ associated to the filtered representation $(\rho,F)$ obtained from $(\rho,K)$, and the space $\Higgs(\rho,K)$ of Higgs forms introduced above. Both spaces are finitely generated graded modules over the ring $M(\Gamma)$ of holomorphic modular forms for $\Gamma$. In the special case that $\Gamma = \PSL_2(\ZZ)$, both modules are free of rank $\dim \rho$ over $M(\Gamma)$. If $\rho$ is unitary and the filtration that $K$ induces on $\rho$ is trivial, then $M(\rho,F)$ and $\Higgs(\rho,K)$ are in fact \emph{equal}, not just isomorphic. There is no reason why these spaces must be isomorphic in general. Since the three-term inequality for modular forms on $\PSL_2(\ZZ)$ is a statement about the weights of a graded basis for $M(\rho,F)$, it is not a priori clear how to use nonabelian Hodge theory to study $M(\rho,F)$. In principle the module $M(\rho,F)$ of interest could look quite different from the corresponding module $\Higgs(\rho,K)$ on the Dolbeault side of the nonabelian Hodge correspondence. The solution is that while nonabelian Hodge theory need not preserve the isomorphism type of the holomorphic bundles involved in its formulation, it does preserve the isomorphism type of their associated cohomology groups (see Section 12 of \cite{Biquard} for the noncompact case, or the earlier paper \cite{Simpson2} for the compact case). It turns out that this is sufficient to show that the three-term inequality is true for all filtered regular connections if and only if it is true for all filtered regular Higgs bundles. Thus, when proving the three-term inequality, we may as well work within the category of filtered regular Higgs bundles. In the next section we describe a class of regular Higgs bundles that satisfy the three-term inequality, but which do not correspond to unitary representations of $\PSL_2(\ZZ)$.

\section{New cases of the three-term inequality for $\PSL_2(\ZZ)$}
\label{s:application}

\subsection{Maximally-decomposed variations of Hodge structure: a proof}
\label{s:newcase}

In \cite{FrancMason}, the three-term inequality for irreducible unitary representations of $\PSL_2(\ZZ)$ was established using the positivity of the corresponding regular connections. Unitary representations endowed with trivial filtrations correspond to Higgs bundles with Higgs field equal to zero. In the next Proposition we establish a new case of the conjecture for a simple type of variation of Hodge structure on $X$, in which the filtered Higgs bundle has a distinguished decomposition into an ordered list of sub-line bundles and with the Higgs field shifting the list by $-1$.

Stability of the Higgs bundle will play a key role, and it suggests that the three-term inequality should be  interpreted as a natural manifestation of stabililty.  Results possessing a similar flavour already exist in the Higgs bundle literature.  For instance, the so-called ``co-Higgs bundles'' on the ordinary projective line $\mathbf{P}^1$, whose Higgs fields take values in the anticanonical line bundle $\mathcal O(2)$ --- which, incidentally, bear a formal similarity to the objects below --- have Grothendieck splitting numbers that are forced by slope stability to satisfy a certain inequality \cite{Rayan}.  In general, one can view constructing a stack of Higgs bundles with fixed topological data (rank, degree, etc.) as an unbounded partition problem in which degrees of sub-sheaves are (possibly negative) integer parts of the total degree.  Passing to the moduli space via stability intervenes to make the problem bounded and, hence, well-posed.   For more on the combinatorics of Higgs bundles and the stability condition, see \cite[\S5]{RayanTopComb}.

\begin{thm}\label{PropMaxDecVar}
  Let $(E,\theta)$ denote a stable filtered regular Higgs bundle on $X$.  Assume that it is a complex variation of Hodge structure, such that $\bar E$, the zero-th filtered piece of $E$, decomposes into line bundles $\bar E \cong \bigoplus_{j=1}^n \cO(r_n)$ such that for each $1\leq j\leq n-1$, the Higgs field $\theta$ restricts to a map
  \[
 \theta \colon \cO(r_{j}) \to \cO(r_{j+1})\otimes \Omega^1_X(\infty) \cong \cO(r_{j+1}+2).
  \]
  If $m_r$ denotes the multiplicity of $r$ among the roots $r_j$ of $\bar E$, then
  \[
  m_r \leq m_{r-2}+m_{r+2}
\]
for all $r$.
\end{thm}
\begin{proof}
  Stability forces all of the maps $\theta$ among the line bundles to be nonzero, save for $\theta|_{\cO(r_n)} = 0$, and all of the roots $r_j$ must have the same parity. Since $H^0(X,\cO(r)) = 0$ if $r < 0$ or $r = 2$, it follows that since $\theta \colon \cO(r_{j}) \to \cO(r_{j+1}+2)$ is nonzero, we must have $r_{j+1}+2-r_{j} \geq 0$ (there are no holomorphic modular forms of negative weight) and $r_{j+1}+2-r_{j} \neq 2$ (there are no holomorphic modular forms of weight $2$ and level one). Therefore, if we consider the graph $G$ obtained by joining the points $(j,r_j)$ by straight lines, then $G$ has the following properties:
  \begin{enumerate}
  \item it has no horizontal segments;
  \item it can drop at most by steps of $2$;
    \item it rises by even integers.
  \end{enumerate}
  
    We wish to show that if we intersect $G$ with the horizontal lines at height $r$, $r-2$ and $r+2$, then the number of points with integer coordinates on the line at height $r$ is bounded by the sums of the number of points on the lines at heights $r+2$ and $r-2$. To see that this is so, label the regions between intersection points with the line at height $r$ with either an $A$, $B$, or $C$. The $A$ and $B$ indicate that $G$ is either entirely above or entirely below the line at height $r$ in the given region. Since $G$ drops by at most two at each step, the only other possibility is that $G$ starts below the line at height $r$, but then rises up by an even amount without hitting the line at height $r$ at an integer lattice point. Label such regions by $C$. Notice that after a region labeled by $C$, the next time $G$ intersects the line at height $r$, it must do so at an integer lattice point from above, as $G$ decreases at each step only by $2$. 

    Consider walking along $G$ from left to right. At a point of intersection with the line at height $r$, if the region to the right is labelled $B$, we may pair the intersection point at height $r$ with the following intersection point that must be at height $r-2$, since the graph decreases exactly by $2$. If instead the region is labelled $A$, then we may travel along the region up until the next intersection point at height $r$. This next point is approached from above. We thus pair the point under consideration with the point at height $r+2$ to the left of the \emph{next} intersection point at height $r$. If the region is labelled $C$, we can in fact pair our point with the point to the right of it at height $r-2$, or the point to the left of the next intersection point, necessarily at height $r+2$.

    It remains to handle the rightmost intersection point. Suppose first that $r < r_1$. Then to the left of the leftmost intersection point at height $r$, the graph must drop down to meet that intersection point. That means that the line at height $r+2$ contributes one more intersection point to $m_{r+2}$ to the left of the leftmost intersection point at height $r$, and hence $m_r \leq m_{r-2}+m_{r+2}$.

    Finally, suppose instead that $r \geq r_1$.  Stability implies that $r_n < r_1$. Thus, $G$ must eventually drop below $r_1$ after the rightmost intersection point at height $r \geq r_1$. Hence the line at height $r-2$ will contribute an extra point of intersection, contributing $+1$ to $m_{r-2}$, and this again proves that $m_r \leq m_{r-2}+m_{r+2}$.\end{proof}

\subsection{Nilpotent Higgs bundles: a sketch of an idea}  As noted previously, one of the nice features of working with Higgs bundles is the confluence of geometric structures on their moduli space that can be used to study properties of Higgs bundles that remain invariant in a family.  With this in mind, one strategy for proving the three-term inequality for a larger class of filtered regular Higgs bundles is to use the differential and symplectic topology of the moduli space to reduce the problem to the complex variations of Hodge structure in Proposition \ref{PropMaxDecVar}.  The key fact here is that the complex variations of Hodge stucture are fixed points of a $\mathbf{C}^*$-action on the moduli space of stable filtered regular Higgs bundles.  This action, whose compact analogue was first exploited by Hitchin in \cite{Hitchin}, simply rescales the Higgs field:$$\theta\stackrel{\lambda}{\mapsto}\lambda\theta.$$ This action interacts with the structure of the moduli space via the \emph{Hitchin map} that sends each filtered Higgs bundle to the tuple of coefficients of the characteristic polynomial of $\theta$.  In this sense, the Hitchin map is a generalized Chevalley morphism for the Lie algebra of the structure group of the filtered bundles.  The target space of this map is an affine space and the fibres are called \emph{Hitchin fibres}.\footnote{As a side note, it is now known from the work of Laumon, Ng\^o, Chaudouard, and others (for example, \cite{ChaudouardLaumon,LaumonNgo,Ngo}) that the Hitchin fibration is closely related to the affine Springer fibration in geometric representation theory, which is yet another avenue along which Higgs bundles enter automorphic representation theory.}   The complex variations of Hodge structure have nilpotent $\theta$ (meaning as per usual that $\theta^k=0$ for some $k\geq1$), and so they lie in a special fibre of the Hitchin map, called the \emph{nilpotent cone} or \emph{core}, which we signify by $\mathcal N$.  Now, the particular variations considered in Proposition \ref{PropMaxDecVar} correspond in the Higgs bundle literature to fixed points of type $(1,...,1)$, where the $1$'s are simply recording the fact that each summand in the Hodge structure is of rank $1$.  In general, the type of a fixed point is some ordered partition of the total rank, $n$.   The various types of fixed points are organized as a poset within the cone according to the value of a real-valued functional defined on the moduli space.  Typically, this is functional is a nondegenerate function $f\geq0$, called the \emph{Morse function}, which in this case is some scalar multiple of the $L^2$-norm of $\theta$.  This norm descends to the moduli space from an infinite-dimensional affine space of smooth connections via the K\"ahler quotient construction \cite{BiswasGothenLogares} and is a moment map for the compact group $U(1)\hookrightarrow\mathbf{C}^*$, which has the same fixed points as the noncompact action.  In general, $\mathcal N$ is a highly singular variety with multiple irreducible components.  One expects that the type $(1,...,1)$ fixed points are local maxima of $f|_{\mathcal N}$.  In other words, to each irreducible component, there should be an ordered $n$-tuple of integers $m_j$, and vice-versa.  (Different orderings may be possible and correspond to distinct components.)  The downward gradient flow of $f$ from the type $(1,...,1)$ fixed points eventually captures all of $\mathcal N$. The flow trajectories are determined by the positive weight space of the normal bundle to $\mathcal N$, which in turn is determined by the positive eigenvalues of $\mbox{Hess}(f)$.  Some of these directions correspond to deformations of $\theta$ that leave the bundle $\bar E$ invariant, while others transform the holomorphic structure on $\bar E$ alone. This former type of deformation reorganizes the line bundles in $\bar E$ but preserves the numbers $m_j$.

The challenge here is to show that:\begin{enumerate}\item for any $(E,\theta)\in\mathcal N$, its splitting into line bundles is a permutation of those appearing in a connected component of the $(1,...,1)$ fixed point locus and that $(E,\phi)$ is connected to that locus via the Morse flow;\item all reorganizations of the line bundles coming from the $\theta$-deforming flow trajectories either continue to satisfy the three-term inequality (by virtue of having satisfied it for the type $(1,...,1)$ fixed points) or become unstable;\item a version of Morse-Bott or Bia{\l}ynicki-Birula theory, which have both been developed for classical moduli of Higgs bundles (for example, \cite{Wilkin,HauselRodriguez}) can be correctly adapted to the moduli space in this setting, in order to ensure that the gradient flow proceeds normally.\end{enumerate} 

Regarding the last point, these theories, which are usually deployed to detect the topology of a space, are sensitive to singularities and other defects that can interrupt flow lines.  Although the cone itself is singular, the smoothness of the total space of the moduli space is usually sufficient to guarantee sensible results.  Despite the technical wrinkles, these ideas suggest a potentially fruitful avenue to proving the three-term inequality for a large number of cases.  With $\mathcal N$ under control, one could then examine flow lines into the cone from the rest of the moduli space, potentially establishing the three-term inequality for less restrictive Higgs bundles and hence for further types of representations.

\bibliographystyle{plain}
\bibliography{refs}
\end{document}